\documentclass{amsart}%
\usepackage{amsmath}
\usepackage{color}
\usepackage{amssymb}
\usepackage{amsfonts}
\usepackage{graphicx}%

\newtheorem{theorem}{Theorem}[section]
\theoremstyle{plain}

\newtheorem{corollary}[theorem]{Corollary}

\newtheorem{lemma}[theorem]{Lemma}

\newtheorem{step}{Step}
\numberwithin{equation}{section}
\newcommand{\re}{\mathbb{R}}

\def \C{\mathbb C}

\def \R{\mathbb R}

\def \S{\mathbb S}

\def \N{\mathbb N}

\def \H{\mathbb H}
\begin{document}
\title[Supercritical elliptic problems on the round sphere]{Supercritical elliptic problems on the round sphere and nodal solutions to the Yamabe problem in projective spaces}
\author{Juan Carlos Fern\'{a}ndez}
\address{Departamento de Matem\'{a}ticas, Facultad de Ciencias, Universidad Nacional Aut\'{o}noma de M\'{e}xico, CP 04510, M\'{e}xico}
\email{jcfmor@ciencias.unam.mx}
\author{Oscar Palmas}
\address{Departamento de Matem\'{a}ticas, Facultad de Ciencias, Universidad Nacional Aut\'{o}noma de M\'{e}xico, CP 04510, M\'{e}xico}
\email{oscar.palmas@ciencias.unam.mx}
\author{Jimmy Petean}
\address{Centro de Investigaci\'{o}n en Matem\'{a}ticas, CIMAT, Calle Jalisco s/n, 36023 Guanajuato, Guanajuato, M\'{e}xico}
\email{jimmy@cimat.mx}%
\thanks{J.C. Fern\'{a}ndez was supported by a postdoctoral fellowship from DGAPA-UNAM}
\thanks{O. Palmas was partially supported by UNAM under Project PAPIIT-DGAPA IN115119.}
\date{\today}

\begin{abstract}
	
Given an isoparametric function $f$ on the $n$-dimensional round sphere, we consider functions of the form $u=w\circ f$ to reduce the semilinear elliptic problem 
\[
-\Delta_{g_0}u+\lambda u=\lambda\left\vert u\right\vert ^{p-1}u\qquad\text{ on }\mathbb{S}^n
\]
with $\lambda>0$ and $1<p$, into a singular ODE in $[0,\pi]$ of the form $w'' + \frac{h(r)}{\sin r} w' + \frac{\lambda}{\ell^2}\left(\vert w\vert^{p-1}w - w\right)=0$, where $h$ is an strictly decreasing function having exactly one zero in this interval and $\ell$ is a geometric constant. Using a double shooting method, together with a result for oscillating solutions to this kind of ODE, we obtain a sequence of sign-changing solutions to the first problem which are constant on the isoparametric hypersurfaces associated to $f$ and blowing-up at one or two of the focal submanifolds generating the isoparametric family. Our methods apply also when $p>\frac{n+2}{n-2}$, i.e., in the supercritical case. Moreover,  using a reduction via harmonic morphisms, we prove existence and multiplicity of sign-changing solutions to the Yamabe problem on the complex and quaternionic space, having a finite disjoint union of isoparametric hipersurfaces as regular level sets.

\bigskip 

\textsc{Key words: } Singular ODE; Yamabe problem; nodal solution; isoparametric hypersurfaces; shooting method; supercritical elliptic problem.

\textsc{2010 MSC: } 34B16, 35B06, 35B33, 35B44 53C21, 58E20, 58J05.

\bigskip

\end{abstract}
\maketitle

%%%%%%%%%%%%%%%%%%%%%%%%%%%%%%%%%%%%%%%%%%%%%%%%%%%%%%%%%%%%%%%%%%%%%%%%%%%%%%%%%%%%%%%
\section{\textbf{Introduction}}

Let $(M,g)$ be a closed (compact without boundary) Riemannian manifold of dimension $n\geq 3$. We will consider the Yamabe type equations
\begin{equation}\label{Eq:Main}
-\Delta_{g}u+\lambda u=\mu\left\vert u\right\vert ^{p-1}u\qquad\text{ on }M
\end{equation}
where $\lambda\in\mathcal{C}^\infty(M)$, $\mu\in\R$  and $p>1$. In case $\lambda=R_g$ is the scalar curvature and $p=p_n:=\frac{n+2}{n-2}$ is the critical Sobolev exponent, equation \eqref{Eq:Main} is  the well known Yamabe equation, widely studied in the last 50 years (see, for example, \cite{Au,BrMa,cfer2} and the references therein). When $p<\frac{n+2}{n-2}$, we will say that the equation \eqref{Eq:Main} is subcritical and we will call it supercritical if $p>\frac{n+2}{n-2}$. In the subcritical case, as the Sobolev embedding $H^1(M,g)\hookrightarrow L^p(M,g)$ is compact, the existence of positive and sign-changing solutions can be obtained using standard variational methods \cite{St,Wi}.  When $M=\Omega$ is a bounded domain of $\R^{n+1}$ with smooth boundary, there has been recent progress in handling supercritical exponent problems like \eqref{Eq:Main}. A fruitful approach consists in reducing the supercritical problem to a more general elliptic critical or subcritical problem, either by considering rotational symmetries or by means of maps preserving the Laplace operator or by a combination of both, see \cite{ClPi} and the references therein. In case of closed Riemannian manifolds, these reduction methods also apply and have been combined with the Lyapunov-Schmidt reduction method in order to obtain sequences of positive and sign-changing solutions to similar supercritical problems, such that they blow-up or concentrate at minimal submanifolds of $M$ \cite{ClGhMi,DeMuPi,GhMiPi,MiPiVe,PiVa}.

The main interest of this paper is to seek for sign-changing solutions (also called \emph{nodal solutions}) to the problem 
\begin{equation}\label{Eq:Yamabe sphere}
-\Delta_{g_0}u+\lambda u=\lambda\left\vert u\right\vert ^{p-1}u\qquad\text{ on }(\mathbb{S}^n,g_0).
\end{equation}
when $p$ is either subcritical or supercritical. Here $g_0$ denotes the round metric and we will assume from now on that $\lambda>0$ is constant. When $p=p_n$, \eqref{Eq:Yamabe sphere} is a renormalization of the Yamabe problem \eqref{Eq:Main} and this kind of solutions have been studied in \cite{Cl,cfer2,dpmpp,dpmpp2,MuWe} and more recently in \cite{fp,MeMuWe,pv}. The slightly subcritical case has been studied in \cite{RoVe}, where the authors obtained multiplicity of nodal solutions blowing-up at points, while the general subcritical case has been studied in \cite{hp} and in \cite{BrLi}. The method introduced in \cite{hp} allowed the authors to obtain more information about the qualitative behavior of the solutions, for they showed the existence of an infinite number of non constant positive solutions having prescribed level sets in terms of isoparametric hypersurfaces. This method has been further generalized to supercritical exponents in \cite{BeJuPe} to prove a similar result on general closed Riemannian manifolds, including the round sphere. Other results concerning the existence and concentration of positive solutions along minimal submanifolds for the supercritical and slightly supercritical can be found in \cite{DeMuPi,MiPiVe}. However, very little is known about the existence, multiplicity and blow-up of nodal solutions for the supercritical problem on the sphere \eqref{Eq:Yamabe sphere}. One of the few results known by the authors is given in \cite{h}, where the existence of at least one sign-changing solution to the supercritical problem was settled. In this direction we will follow and generalize the ideas introduced in \cite{fp} to obtain an infinite number of nodal solutions to the supercritical and subcritical problem \eqref{Eq:Yamabe sphere}, having as level and critical sets isoparametric hypersurfaces and its focal submanifolds. 

To state our main result and to describe the method, we briefly recall some aspects of the theory of isoparametric functions and hypersurfaces. For the details, we refer the reader to \cite{bco,cr}. A smooth function $f:(M,g)\rightarrow\R$ is isoparametric if there exist smooth functions $a,b:\R\rightarrow\R$ such that
\begin{equation}\label{Eq:Isoparametric}
\vert \nabla f\vert_g^2=b(f)\quad\text{and}\quad\Delta f=a(f).
\end{equation}
The regular level sets of $f$ are called isoparametric hypersurfaces.

The theory of isoparametric hypersurfaces in the round sphere $(\S^n,g_0)$ is very rich and it is a vast research topic. In this case, isoparametric hypersurfaces coincide with the hypersurfaces of constant principal curvatures. Its classification began with E. Cartan \cite{c} and it is still an open problem, see \cite{ch,miy2,miy3} and the references therein. Some major progresses in the theory were made by Cartan himself, H. F. M\"{u}nzner \cite{m1,m2} and D. Ferus, H. Karcher and M\"{u}nzner \cite{fkm}. Given an isoparametric hypersurface, there exist a huge number of isoparametric functions having it as level hypersurface, for if $f:\S^n\rightarrow\R$ is isoparametric, $\nu:Im(f)\rightarrow\R$ is monotone and $\alpha\in\R\smallsetminus\{0\}$, then $\alpha(\nu\circ f)$ is again isoparametric.  However, there are ``canonical'' isoparametric functions, which are obtained by restricting Cartan-M\"{u}nzner polynomials to the sphere \cite[Section 3.5]{cr}. They are well understood and have some nice properties. For instance, if $f:\S^n\rightarrow\R$ is obtained in this way, then $\text{Im}f=[-1,1]$, the inverse image of a regular value is a connected isoparametric hypersurface, its only critical values are $t=\pm 1$ and the functions $a$ and $b$ defined in \eqref{Eq:Isoparametric} can be written explicitly. To give these explicit expressions, let $\ell$ be the number of distinct principal curvatures of the level sets of $f$. M\"{u}nzner showed that $\ell\in\{1,2,3,4,6\}$ and if $\ell$ is odd, then all the multiplicities of the principal curvatures are the same, while if $\ell$ is even, there exist, at most, two different multiplicities $m_{-}$ and $m_{+}$ with $1\leq m_{-},m_{+}\leq n-1$. With this notation, if $\Delta_{g_0} f=a(f)$ and $\vert\nabla f\vert_{g_0}=b(f)$, then
\[
a(t)=-\ell(n+\ell-1)t+\frac{\ell^2(m_{+}-m_{-})}{2}\quad\text{and}\quad b(t)=-\ell^2t^2+\ell^2.
\]
The sets $M_{-}:=f^{-1}(-1)$ and $M_{+}:=f^{-1}(1)$ are smooth submanifolds of $\S^n$ of dimension $n_{-}=(n-1)-m_{-}$ and $n_{+}:=(n-1)-m_{+}$, called focal submanifolds, see \cite[Section 2.4]{cr}. The main feature of these submanifolds is that every isoparametric hypersurface is a tube around $M_{-}$ and $M_{+}$. 

If $u$ denotes a sign-changing smooth function defined on a Riemannian manifold, we define the nodal and the critical sets of $u$ to be the sets $\{u=0\}$ and $\{\nabla u = 0\}$, respectively. We state the main result of this paper, which generalizes Theorem 1.3 in \cite{fp}.

\begin{theorem}\label{Th:supercritical} Let $S\subset \S^n$ be an isoparametric hypersurface and let $n_{-}\leq n_{+}$ be the dimensions of its corresponding focal submanifolds. Then, for any  $\lambda>0$, any $k\in\N$ and any $p\in (1,\frac{n-n_{+}+2}{n-n_{+}-2})$, equation \eqref{Eq:Yamabe sphere} admits a  nodal solution $u_k$ such that its nodal set has at least $k$ connected components, each of them being  an isoparametric hypersurface diffeomorphic to $S$. The critical set of $u_k$ consists in the focal submanifolds $M_{-}$ and $M_{+}$ and, at least, $k-1$ isoparametric hypersurfaces diffeomorphic to $S$. Moreover, the solutions $u_k$ satisfy 
\begin{equation}\label{Eq:Blow-up}
	\lim_{k\rightarrow\infty}\vert u_k(x)\vert=\infty,
\end{equation}
for every $x\in M_-$ or for every $x\in M_+$.
\end{theorem}

Here the numbers $\frac{n-n_{\pm}+2}{n-n_{\pm}-2}\geq p_n$ are just the critical Sobolev exponents in dimensions $n-n_\pm$. Our Theorem improves the existence result stated by Henry in \cite{h}, giving an infinite number of distinct solutions instead of one. It also extends the multiplicity result in \cite{fp} to the subcritical and supercritical exponents. However, this last result gives a better description of the nodal set of the solutions. We strongly believe that a refinement of our methods may give a prescribed number of connected components for the nodal sets of the sign-changing solutions to problem \eqref{Eq:Main}, as in Theorem 1.2 in \cite{fp}. 

The last assumption of Theorem \ref{Th:supercritical} says that the sequence $(u_k)$ is not compact with the $C^0$ topology and that the blow-up occurs on one of the focal submanifolds, which are minimal submanifolds of the sphere \cite{cr}. Other noncompactness phenomenon of the same nature appears in the solutions to the critical Yamabe problem obtained in  \cite{dpmpp}, where the blow-up occurs at a single point. However, it was recently proved by Premoselli and V\'{e}tois that this sequence of solutions is uniformly bounded from below, but not from above \cite{pv}. This does not holds true in general, as we state next.

\begin{corollary}\label{Cor:Blowup}
Let $S$ be an isoparametric hypersurface with focal submanifolds $M_-$ and $M_+$ satisfying $\dim M_{\pm}>0$. Then there exists a sequence $(u_k)$ of sign-changing solutions to the Yamabe problem on the sphere
\begin{equation}\label{Eq:CritYamabe sphere}
-\Delta_{g_0}u+\frac{n(n-2)}{4} u=\frac{n(n-2)}{4}\left\vert u\right\vert ^{p_n-1}u\qquad\text{ on }\mathbb{S}^n,
\end{equation}
satisfying that
\[
\lim_{k\rightarrow\infty}u_k(x)=\infty \text{ for every }x\in M_-\ \text{ and } \ \lim_{k\rightarrow\infty}u_k(x)=-\infty \text{ for every }x\in M_+.
\]
\end{corollary}

As another consequence of Theorem \ref{Th:supercritical}, we obtain a multiplicity result for the Yamabe problem on projective spaces.

\begin{corollary}\label{Cor:YamabeProjective}
Let $(M,g)$ be the complex projective space $\C P^m$ or the quaternionic projective space $\mathbb{H}P^m$ endowed with their canonical metric. Then, if  $j=2$ in case of $\C P^m$ and $j=4$ in case of $\H P^m$, for every $k\in\N$, the Yamabe equation 
\begin{equation}\label{Eq:Yamabe projective}
-\frac{4(jm-1)}{jm-2}\Delta_{g}v+R_g v=\left\vert v\right\vert ^{p_{jm}-1}v\qquad\text{ on }(M,g),
\end{equation}
admits a sequence of sign-changing solutions $(u_k)$ such that the regular level sets of $u_k$ consist of isoparametric hypersurfaces in $(M,g)$ and
\begin{equation}\label{Eq:Blow-upProjective}
\lim_{k\rightarrow\infty}\max_{x\in M}\vert u_k(x)\vert=\infty.
\end{equation}
\end{corollary} 

We describe briefly the method we shall use in order to prove Theorem \ref{Th:supercritical}. The details will be given in Section \ref{Sec:Double Shooting} and Section \ref{Sec:Projective}.

Let $f:\S^n\rightarrow\R$ be an isoparametric function obtained as the restriction of a Cartan-M\"{u}nzner polynomial, and let $\ell$, $m_{-}$ and $m_{+}$ be the number of principal curvatures and the multiplicities associated to the isoparametric hypersurfaces that $f$ defines, as it was explained before. Then, it is easy to see that $z:[-1,1]\rightarrow\R$ is a solution to the problem
\begin{equation}\label{Eq:Subcritical Yamabe ODE}
b(t)z'' + a(t)z' + \lambda [\vert z\vert^{p-1}z-z] = 0 \quad \text{ on }\  [-1,1],
\end{equation}
with $a(t):=-\ell(n+\ell-1)t+\frac{\ell^2(m_{+}-m_{-})}{2}$ and $b(t):=-\ell^2t^2+\ell^2$, if and only if $u=z\circ f$ is a solution to the problem \eqref{Eq:Yamabe sphere} (Cf. \cite{fp}). Therefore, if $u=z\circ f$ is a solution to \eqref{Eq:Yamabe sphere}, its regular level sets and the set of its critical points are conformed by isoparametric hypersurfaces and focal submanifolds. We can simplify equation \eqref{Eq:Subcritical Yamabe ODE} even more  by considering the new variable $w(r)=z(\cos r)$, and, in this way,  solving \eqref{Eq:Subcritical Yamabe ODE} is equivalent to solving the singular ODE
\begin{equation}\label{Eq:Subcritical radial Yamabe ODE}
w'' + \frac{h(r)}{\sin r} w' + \frac{\lambda}{\ell^2}\left(\vert w\vert^{p-1}w - w\right)=0 \text{ on } [0,\pi],
\end{equation}
where $h(r)=\frac{n-1}{\ell}\cos r - \frac{m_{+}-m_{-}}{2}=\frac{m_{-}+m_{+}}{\ell}\cos r - \frac{m_{+}-m_{-}}{2}$.

Observe that the natural boundary conditions associated to this problem are given by $w'(0)=w'(\pi)=0$. Theorem \ref{Th:supercritical} will be a consequence of the following one.

\begin{theorem}\label{Th:Nodal Yamabe ODE}	For any $p\in (1,\frac{n-n_{-}+2}{n-n_{-}-2})$ and any $k\in\N$, the equation \eqref{Eq:Subcritical radial Yamabe ODE} with boundary conditions $w'(0)=w'(\pi)=0$ admits a sign changing solution $w_k$ having at least $k$ isolated zeroes in $[0,\pi]$ and at least $k+1$ isolated critical points.
\end{theorem}

We will prove this theorem in Section \ref{Sec:Projective}.

The function $h$ appearing in Equation \eqref{Eq:Subcritical radial Yamabe ODE} has a unique zero $a_0 \in (0,\pi )$. To prove Theorem \ref{Th:Nodal Yamabe ODE} we will use the double shooting method developed in \cite{fp}, which consists in considering the solutions $w_d$, $\widetilde{w}_c$ of
Equation \eqref{Eq:Subcritical radial Yamabe ODE} with initial conditions $w_d'(0) = \widetilde{w}_c '(\pi )=0$,
$w_d (0)=d$, $\widetilde{w}_c (\pi )=c$ and consider the maps $I(d) = (w_d (a_0 ), w_d ' (a_0 ))$ and
$J(c) =(\widetilde{w}_c (a_0 ), \widetilde{w}_c ' (a_0 ))$. If $I(c) = J(d)$, then
$w_d = \widetilde{w}_c$ is a solution of Equation \eqref{Eq:Subcritical radial Yamabe ODE} with $w_d'(0) = w_d '(\pi )=0$, as one can readily see.
To understand the intersections of the curves $I, J$ one needs information of the functions $w_d$, $\widetilde{w}_c$. In the next section we will prove that, for large $d$ and $c$, these functions have many zeroes close to $0$ and $\pi$ (respectively) and then, generalizing an argument based on a Pohozaev-type identity and presented in \cite{cf}, we will prove that $\vert I(d)\vert,\vert J(c)\vert\rightarrow\infty$ as $c,d\rightarrow\infty$. These two results will allow us to conclude that the curves $I$ and $J$ behave as spirals rotating in opposite directions and from this we will obtain the intersections needed to solve the double shooting problem.

%%%%%%%%%%%%%%%%%%%%%%%%%%%%%%%%%%%%%%%%%%%%%%%%%%%%%%%%%%%%%%%%
%%%%%%%%%%%%%%%%%%%%%%%%%%%%%%%%%%%%%%%%%%%%%%%%%%%%%%%%%%%%%%%%

\section{\textbf{Double shooting and the proof of Theorem \ref{Th:Nodal Yamabe ODE}.}}\label{Sec:Double Shooting}

We now develop the double shooting method used to prove Theorem \ref{Th:Nodal Yamabe ODE}. First, observe that the function $h$ defined in \eqref{Eq:Subcritical radial Yamabe ODE} satisfies $h(0)=m_{-}$, $h(\pi ) = -m_{+}$, it is strictly decreasing, has a unique zero $a_0\in(0,\pi)$ and $h(r)>0$ in $[0,a_0)$, while $h(r)<0$ in $(a_0,\pi]$. Moreover, the function $\widetilde{h}(r):=-h(\pi-r)=\frac{m_{-} + m_{+}}{2} \cos r +\frac{m_{+} - m_{-}}{2}$ has the same properties with $m_{-}$ and $m_{+}$ interchanged and a unique zero at $\pi-a_0$. To handle both singularities in \eqref{Eq:Subcritical radial Yamabe ODE} at the same time, the strategy is to shoot solutions from each of them and expect that, for some suitable initial and final conditions, the solutions coincide. That is, we consider the initial value problem
\begin{equation}\label{Eq:Singular forward}
\left\{\begin{tabular}{cc}
$w_{i}''(r) + \frac{h(r)}{\sin r} w_i'(r) + \frac{\lambda}{ \ell^2 }( |w_i(r)|^{p-1} w_i -w_i) =0$& in $[0,a_0]$,\\
$w_i(0)=d,\ \  w_i'(0)=0,$ &
\end{tabular}\right.
\end{equation}
and the ``final'' value problem
\begin{equation}\label{Eq:Singular backward}
\left\{\begin{tabular}{cc}
$w_f''(r) + \frac{h(r)}{\sin r} w_f'(r) + \frac{\lambda}{ \ell^2}( |w_f(r)|^{p-1} w_f -w_f) =0$& in $[a_0,\pi]$,\\
$w_f(\pi)=c, \ \  w_f'(\pi)=0,$ &
\end{tabular}\right.
\end{equation}
looking for initial and final conditions $d$ and $c$ such that $w_i(a_0,d)=w_f(a_0,c)$ and $w'_i(a_0,d)=w'_f(a_0,c)$. Hence, by uniqueness of the solution, we would have a well defined solution to problem \eqref{Eq:Subcritical radial Yamabe ODE} given by $w(r)=w_i(r,d)$ if $r\in[0,a_0]$ and $w(r)=w_f(r,c)$ if $r\in[a_0,\pi]$. To construct the solutions with an arbitrarily large number of zeroes, we will need to use that the number of zeroes before and after $a_0$ goes to infinity as $\vert d\vert,\vert c\vert\rightarrow\infty$.

Actually, problem \eqref{Eq:Singular backward} can be written as an initial condition problem having the form of \eqref{Eq:Singular forward}. Indeed, if we consider the function $\widetilde{h}(r)$ defined above, then $w_f$ solves \eqref{Eq:Singular backward} if and only if $\omega(r)=w_f(\pi-r)$ solves the initial value problem
\begin{equation}\label{Eq:Singular backward equivalent}
\left\{\begin{tabular}{cc}
$\omega''(r) + \frac{\widetilde{h}(r)}{\sin r} \omega'(r) + \frac{\lambda}{ \ell^2}( |\omega(r)|^{p-1} \omega -\omega) =0$& in $[0,\pi-a_0]$,\\
$\omega(0)=c,\ \  \omega'(0)=0,$ &
\end{tabular}\right.
\end{equation}
So, in order to understand problem \eqref{Eq:Singular forward} it is enough to consider problem \eqref{Eq:Singular backward}.

In what follows, we will consider the more general equation

\begin{equation}\label{Eq:General singular}
\left\{\begin{tabular}{cc}
$w''(r) + \frac{H(r)}{r} w'(r) + \mu( |w(r)|^{p-1} w -w) =0$ & in $[0,A]$\\
$w(0)=d, \ \ w'(0)=0,$ &
\end{tabular}\right.
\end{equation}
where $A>0$, $\mu>0$ and $H$ is a non negative $C^1$ function defined in the interval $[0,A]$. Notice that equations \eqref{Eq:Singular forward} and \eqref{Eq:Singular backward equivalent} are special cases of the former by taking $\mu=\frac{\lambda}{\ell^2}$, $H(r)=\frac{h(r)r}{\sin r}$ in $[0,A]$ with $A<a_0$ and $H(r)=\frac{\widetilde{h}(r)r}{\sin r}$ in $[0,A]$ with $A<\pi-a_0$. Observe that now we are just dealing with a single singularity at $r=0$.

A standard contraction map argument (Cf. \cite{fp,k}) yields the existence and uniqueness of the solutions to equation \eqref{Eq:General singular} with initial conditions $w(0)=d\in\R$ and $w'(0)=0$, depending continuously on $d$. For $d> 0$, let $w_d:=w(\cdot,d)$ be the solution with initial values $w_{d } (0) =d$ and  $w_{d } ' (0)=0$. To assure the existence of an arbitrarily large number of zeroes, we use the following result, proven in \cite{fp,h}.

\begin{theorem}\label{Prop:prescribed zeroes}
	Suppose  that $H(0) >0$, $p>1$ and that the following inequality
	\begin{equation}\label{Eq: Subcriticality ODE}
	\frac{H(0) +1 }{2} < \frac{p+1}{p-1}
	\end{equation}
	holds true. Then, for any $\varepsilon >0$ and  any positive integer $k$ there exists $D_k > 0$ so that the solution $w_{d }$ of \eqref{Eq:General singular} has at least $k$ zeroes in $(0,\varepsilon )$ for any $d\geq D_k$.
\end{theorem}

In case of equations \eqref{Eq:Singular forward} and \eqref{Eq:Singular backward equivalent}, we have that $H(0)=h(0)=m_{-}$ and $H(0)=\widetilde{h}(0)=m_{+}$ respectively. Taking $1\leq m_{-}\leq m_{+}\leq n-1$, and recalling that $n_{-}=(n-1)-m_{-}$ and $n_{+}=(n-1)-m_{+}$,  inequality 
\begin{equation}\label{Eq:Key inequality}
1<p<\frac{n-n_{+}+2}{n-n_{+}-2}=\frac{m_{+}+3}{m_{+}-1}\leq \frac{m_{-}+3}{m_{-}-1}=\frac{n-n_{-}+2}{n-n_{-}-2}
\end{equation}
implies the validity of inequality \eqref{Eq: Subcriticality ODE} for both equations \eqref{Eq:Singular forward} and \eqref{Eq:Singular backward equivalent}.

Since $p_n\leq \frac{m_{+}+3}{m_{+}-1}$ is true for every $1\leq m_{+}\leq n-1$, we may guarantee the existence of an arbitrary large number of zeroes for equations \eqref{Eq:Singular forward} and \eqref{Eq:Singular backward equivalent} when $p<p_n$, corresponding to the subcritical  Yamabe problem, or when $p_n<p<\frac{m_{+}+3}{m_{+}-1}$, corresponding to the supercritical one. Also observe that inequality \eqref{Eq:Key inequality} is true when $p=p_n$ and $1\leq m_{-},m_{+}<n-1$, a fact used in \cite{fp} to assure the existence of a prescribed number of zeroes to equation \eqref{Eq:Subcritical radial Yamabe ODE} in the critical case. For the rest of this section, we will suppose that $p>1$ satisfies inequality \eqref{Eq:Key inequality}.

Even if the number of zeroes is arbitrarily large, it can  not be infinite as we next show.

\begin{lemma}\label{Cor:No zeroes concentration}
	For $d_\ast>0$, $d_\ast\neq 1$ fixed, the zeroes and critical points of $w_{d_\ast}$ are isolated in $[0,A]$.
\end{lemma}

\begin{proof}
First observe that if $r_0\in[0,A]$ is a zero of $w_{d_\ast}$, then uniqueness of the solution implies that $w_{d_\ast}(r_0)\neq 0$. Therefore $w_{d_\ast}$ is monotone in a neighborhood of $r_0$ and it is an isolated zero. Now, to see that the critical points are isolated, suppose that $w'_{d_\ast}(r_0)=0$ for some $r_0\in[0,A]$. As $d\neq 0,1$, by uniqueness of the solutions, $w_{d_\ast}(r_0)\neq 0$ and this together with equation \eqref{Eq:General singular} implies that $w_{d_\ast}''(r_0)\neq 0$. Without loss of generality, suppose $w''_{d_\ast}>0$. By continuity, there exists $\varepsilon>0$ such that $w''_{d_\ast}(r)>0$ in $(r_0-\varepsilon,r_0+\varepsilon)\cap[0,A]$. Hence $w'_{d_\ast}$ is monotone in $(r_0-\varepsilon,r_0+\varepsilon)\cap[0,A]$ and, therefore, $w_{d_\ast}'(r)\neq 0$ in $(r_0-\varepsilon,r_0+\varepsilon)\cap[0,A]\smallsetminus\{r_0\}$ and $r_0$ is an isolated critical point. 
\end{proof}

For $\varepsilon<A$, take $D_k>0$ as given in Theorem \ref{Prop:prescribed zeroes}. Then, for  $d \geq D_k>0$, $w_d$ has at least $k$ zeroes before $A$. Denote them by $r_1(d)<r_2(d)<\ldots < r_k(d)$.  The following statement holds true.

\begin{lemma}\label{Lemma:Limit zeroes}
	For each $j=1,\ldots,k$,
	\[
	\lim_{d\rightarrow\infty}r_j(d)=0 
	\]
\end{lemma}

\begin{proof}
Let
\[
z_{d } (r): = d ^{-\frac{2}{p-1}} \   w_{d ^{\frac{2}{p-1}}}\left( \frac{r}{d   \sqrt{\lambda} } \right).
\]
For any compact subset $K$ of $[0,A]$, the functions $z_d$ converge $C^1$-uniformly on $K$ to the unique solution of the limit Cauchy problem	
\begin{equation}\label{Eq:Limit problem}
\left\{\begin{tabular}{cc}$v''(r) + \frac{H(0)}{r} v'(r) +  |v(r)|^{p-1} v  =0$ & \text{ in } $[0,\infty)$,\\
$v(0)=1,\ \ v'( 0)=0$ & \end{tabular}\right.
\end{equation}
and as $p>1$ satisfies inequality \eqref{Eq:Key inequality}, $v$ has an infinite number of isolated zeroes in $(0,\infty)$
(see Lemma 3.2 and Theorem 3.3 in \cite{fp} and Proposition 3.6 in \cite{hw}).	

Now observe that
	\[
	w_d(r)=dz_{d^{\frac{p-1}{2}}}\left(\sqrt{\lambda}d^{\frac{p-1}{2}r}\right)\qquad\text{and}\qquad w'_d(r)=\sqrt{\lambda}d^{\frac{p+1}{2}}z'_{d^{\frac{p-1}{2}}}\left( \sqrt{\lambda}d^{\frac{p-1}{2}}r \right).
	\]
Therefore $r$ is a zero of $w_d$ if and only if $\sqrt{\lambda}d^{\frac{p-1}{2}}r$ is a zero of $z_{d^{\frac{p-1}{2}}}$. For each $j=1,\ldots,k$, denote by $r^\ast_j(d)$ the $j$-th zero of $z_{d^{\frac{p-1}{2}}}$ and by $a_j$ the $j$-th zero of the solution $v_0$ to the limit problem \eqref{Eq:Limit problem}. Then $r_j^\ast(d)=\lambda d^{\frac{p-1}{2}}r_j(d)$ and the $C^1$ uniform convergence of $z_d$ to $v_0$ implies that $r_j^\ast(d)\rightarrow a_j$ as $d\rightarrow\infty$ for each $j=1,\ldots,k$. Hence $r_j(d)\rightarrow 0$ as $d\rightarrow\infty$.
\end{proof}

Now we focus on equation \eqref{Eq:Singular forward} in order to do a phase plane analysis. Let $a_0$ be the unique zero of $h$ in $(0,\pi )$. Let $w_d$ be the solution of \eqref{Eq:Singular forward}
with initial conditions $w_d (0)=d$, $w_d' (0)=0$.

Consider the curve $I: \re \rightarrow \re^2$ given by $I(d) = (w_d (a_0 ) , w_d ' (a_0 ) )$. Note that
$I(0) = (0,0)$, $I(-d) =-I(d)$ and $I(d) \neq (0,0)$ if $d\neq 0$. It is then easy to see that we have
a well defined continuous function $\theta : (0, \infty ) \rightarrow \re$ such that $\theta (1)=0$
and $\theta (d)$ gives an angle between $I(d)$ and the positive $x$-axis for any $d>0$. Note that, in a similar way, there is a unique continuous function  $\theta : (-\infty ,0) \rightarrow \re$ such that $\theta (-1) = -\pi$
and $\theta (d)$ gives an angle between $I(d)$ and the positive $x$-axis. Thus, we have that for any
$d>0$, $\theta (-d) = \theta (d) - \pi $. Also notice that $w_d (a_0 )=0$ if and only if $\theta (d) = -\frac{\pi}{2} - k \pi$ for some  integer $k$.

Next we proceed to define a second curve in the phase space, corresponding to the solutions to problem \eqref{Eq:Subcritical radial Yamabe ODE} in $[0,\pi]$ with condition $w'(\pi)=0$. Let $\widetilde{h}(r)=-h(\pi-r)=\frac{m_{-}+m_{+}}{2}\cos r + \frac{m_{+}-m_{-}}{2}$ and consider the  problem \eqref{Eq:Singular backward equivalent}. If $\omega$ is a solution to this problem, then $\widetilde{w}(r):=\omega(\pi-r)$ solves the ``final'' conditions problem \eqref{Eq:Singular backward}. For $c\in\R$, denote by $\widetilde{w}_c$ the solution to the problem \eqref{Eq:Singular backward} and define the map $J(c):=(\widetilde{w}_c(a_0),\widetilde{w}_c'(a_0))$.

In an entirely similar way, $J(1)=(1,0)$, $J(0)=(0,0)$, $J(c) \neq (0,0)$ if $c\neq 0$ and $J(-c)=-J(c)$. So, there is a well defined angle $\vartheta:\R\smallsetminus\{0\}\rightarrow\R$ such that $\vartheta(1)=0$ and
\[
J(c)=(\vert J(c)\vert\cos(\vartheta(c)),\vert J(c)\vert\sin(\vartheta(c)).
\]
Note that $\theta$ and $\vartheta$ run in opposite directions. If $n(d)$ denotes the number of zeroes of $w_d$ before $a_0$ and $N(c)$ the number of zeroes of $\widetilde{w}_c$ after $a_0$, then the angles $\theta$ and $\vartheta$ are related with these numbers by the following formulas, proved in \cite{fp},
\begin{equation}\label{Eq:FormulaZeroes}
n (d) = -\left\lfloor\frac{\theta(d) - \pi /2}{\pi} \right\rfloor -1\quad\text{ and }\quad N(c)= - \left\lfloor \frac{-\vartheta(c) - \pi /2}{\pi}\right\rfloor  -1,
\end{equation}
where, as usual for $x\in\R$, $\lfloor x\rfloor$ denotes the maximum integer such that $\lfloor x\rfloor \leq x$. As a consequence of these formulas and Theorem \ref{Prop:prescribed zeroes}, we have the following limits
\begin{equation}\label{Eq:limit argument}
\lim_{d\rightarrow\infty}\theta(d)=-\infty\quad\text{ and }\quad\lim_{c\rightarrow\infty}\vartheta(c)=\infty.
\end{equation}

We will show that the curves $I$ and $J$ behave like spirals turning in opposite directions. The above limits show that the spirals actually turn. Next we see that the spirals are not bounded. If $w_d$ is a solution to the initial value problem \eqref{Eq:Singular forward} and $w_c$ is a solution to \eqref{Eq:Singular backward}, define
$\rho(r,d):=\sqrt{w_d^2(r)+(w_d')^2(r)}$ and $\varrho(r,c)=\sqrt{w_c^2(r)+(w_c')^2(r)}.$

\begin{lemma}\label{Lemma:Radius}
We have that
	\[
	\lim_{d\rightarrow\infty}\rho(r,d)=\infty,\quad\text{and}\quad 	\lim_{c\rightarrow\infty}\varrho(r,c)=\infty,
	\]
uniformly in $[0,a_0]$ and in $[a_0,\pi]$, respectively.
	\end{lemma}

The proof of this Lemma is technical and we postpone it to the Appendix \ref{Sec:Energy Analysis}.

Consider the functions $\rho(d):=\rho(d,a_0)$, $\varrho(c):=\varrho(c,a_0)$, $\theta(d):=\theta(d,a_0)$ and $\vartheta(c):=\vartheta(c,a_0)$.

Now, define in the radius-argument plane the curves $R:[1,\infty)\rightarrow \R\times\R_{>0}$ and $S:[1,\infty)\rightarrow\R\times\R_{>0}$ given by
\[
R(d):=(\theta(d),\rho(d)),\quad\text{ and }\quad S(c):=(\vartheta(c),\varrho(c))
\]
From uniqueness of the solutions to problems \eqref{Eq:Singular forward} and \eqref{Eq:Singular backward}, the curves $R$ and $S$ are simple and they intersect at the point $(0,1)$.

For each $i,j\in\N$, define
\[
d_i:=\max\{d\;:\; \theta(d)=-i\pi\}\quad \text{ and }\quad c_j:=\max\{c\;:\;\vartheta(c)=j\pi\},
\]
and
\[
\widehat{d}_i:=\min\{d\;:\; \theta(d)=-i\pi\}\quad \text{ and }\quad \widehat{c}_j:=\min\{c\;:\;\vartheta(c)=j\pi\},
\]
These numbers are well defined by \eqref{Eq:limit argument} and they form unbounded sequences by the same limit.  Observe that $\widehat{d}_i$ and $d_i$ are first and last time that the curve $R$ hits the line  $\theta=-i\pi$, respectively, while $\widehat{c}_i$ and $c_j$ correspond to the first and last time that $S$ hits the line $\vartheta=j\pi$. It was also shown in \cite{fp} that for any $c,d>0$, $\theta(d)<\pi/2$ and $\vartheta(c)>-\pi/2$. So, it follows that the curve $R$ is completely contained in $(-\infty,\pi/2)\times\R_{>0}$, while $S$ is contained in $(-\pi/2,\infty)\times\R_{>0}$, and that $R$ restricted to $[d_1,\infty)$ and $S$ restricted to $[c_1,\infty)$ do not intersect. 

We can now prove Theorem \ref{Th:Nodal Yamabe ODE}.

\begin{proof}[Proof of Theorem \ref{Th:Nodal Yamabe ODE}.]
Let $k\in\N$ and for each $i,j\in\N$, set $x_i:=\rho(d_i)$, $\widehat{x}_i:=\rho(\widehat{d}_i),$ $y_j:=\varrho(c_j)$ and $\widehat{y}_j:=\varrho(\widehat{c}_j)$. 
By Lemma \ref{Lemma:Radius}, these sequences are unbounded. Therefore, we can find $i,j>k$ and $\alpha>j$ such that $y_j<\min\{x_i,\widehat{x}_i,x_{\alpha+1-j}\}<\widehat{y}_\alpha$. Observe that the curves $R$ and $S-((\alpha+j)\pi,0)$ restricted to the intervals $[d_i,\widehat{d}_{\alpha+i-j}]$ and $[c_j,\widehat{c}_\alpha]$, respectively, are both contained in $\mathcal{A}_k:=[-(\alpha+i-j)\pi,-i\pi]\times[y_j,\widehat{y}_\alpha]$. As $\widehat{x}_{\alpha+i-j}>y_j$ and $x_i<\widehat{y}_{\alpha}$ and as $R$ restricted to $[d_j,\widehat{d}_{\alpha+j-i}]$ intersects $\mathcal{A}_k$ only at the points $(-(\alpha+j-i)\pi,\widehat{x}_{\alpha+j-i})$ and $(-j\pi,x_j)$, the intermediate value Theorem implies that the curve $R$ must intersect $S-((\alpha+j)\pi,0)$. Let $d_R> 1$ and $c_S>1$ be the points such that $R(d_R)=S(c_S)-((\alpha+j)\pi,0)$. Using the formulas \eqref{Eq:FormulaZeroes}, we can argue as in the proof of Theorem 1.2 in \cite{fp} to conclude that $w_{d_R}=\widetilde{w}_{c_S}$ is a solution to the problem \eqref{Eq:Subcritical radial Yamabe ODE} with exactly $\alpha+j>k$ zeroes and, since $w'_{d_R}(0)=w'_{d_R}(\pi)=0$, it has, at least, $k+1$ critical points by Rolle's Theorem. The fact that the zeroes and that the critical points are isolated follows from Lemma \ref{Cor:No zeroes concentration}.	
\end{proof}

This theorem implies \ref{Th:supercritical} as follows.

\begin{proof}[Proof of Theorem \ref{Th:supercritical}.]
Associated to $S$, there is a Cartan-M\"{u}nzner polynomial such that its restriction to the sphere is an isoparametric function $f:\S^n\rightarrow[-1,1]$. Then equation \eqref{Eq:Yamabe sphere} can be reduced into equation \eqref{Eq:Subcritical radial Yamabe ODE}. By Theorem \ref{Th:Nodal Yamabe ODE}, this equation admits a sign-changing solution $w_k$ having at least $k$ isolated zeroes and at least $k+1$ isolated critical points in $[0,\pi]$. Therefore, $u_k=w_k(\arccos f)$ is a solution to the problem \eqref{Eq:Yamabe sphere} having as regular level sets a disjoint union of connected isoparametric hypersurfaces diffeomorphic to $S$. As $w_k$ has at least $k$ isolated zeroes, then the nodal set $u_k^{-1}(0)$ has at least $k$ connected components, all of them being isoparametric hypersurfaces diffeomorphic to $S$. As $w_k$ has at least $k-1$ isolated critical points in $(0,\pi)$ and as $w_k'(0)=0=w_k'(\pi)$, the critical set of $u_k$ consists in, at least, $k-1$ connected isoparametric submanifolds diffeomorphic to $S$, together with the focal submanifolds $M_{-}=f^{-1}(-1)$ and $M_{+}=f^{-1}(1)$. Finally, by construction of $w_k$, we have that $w_k(0)\rightarrow\infty$ or $\vert w_k(\pi)\vert \rightarrow\infty$ as $k\rightarrow\infty$, for the number of zeroes increases as the initial or final conditions increase. Suppose, without loss of generality, that $w_k(0)\rightarrow\infty$. In this case, as $\arccos\left( f(x)\right)=0$ for all $x\in M_-$ we concluding the limit \eqref{Eq:Blow-up} for every $x\in M_-$.
\end{proof}

Even if Theorem \ref{Th:supercritical} holds for subcritical, critical and supercritical values of $p>1$, the methods developed here do not allow us to prove the existence of solutions such that the final value $w_d(\pi)=c$ is negative. However, in case of the critical exponent $p_n$, we can use the refined version of this theorem given in \cite{fp} to construct a sequence of solutions which is not uniformly bounded from below.

\begin{proof}[Proof of Corollary \ref{Cor:Blowup}]
As the focal manifolds that generate $S$ have positive dimensional, the number of distinct principal curvatures $\ell$ must be bigger that $1$. Let $f$ be the Cartan-M\"{u}nzner polynomial associated to $S$. Then, for any $k\in\N$, Theorem 1.2 in \cite{fp} gives a solution to the Yamabe problem on the sphere \eqref{Eq:CritYamabe sphere} having the form $u_k=\arctan(f(w_{k}))$, where $w_{k}$ is a solution to the problem \eqref{Eq:Subcritical radial Yamabe ODE}, with $p=p_n$, having exactly $k$ zeroes in $[0,\pi]$. Lemma \ref{Lemma:Radius} together with Lemma 4.6 in \cite{fp} imply that the sequences $(x_i)$ and $(y_j)$ are both increasing and unbounded. In this way, the situation in Lemma 4.7 in \cite{fp} can not happen and the number of zeroes before and after $a_0$ must increase as the initial and final conditions, $w_k(0)$ and $\vert w_k(\pi)\vert$ respectively, increase. For $k$ odd, by Lemma \ref{Cor:No zeroes concentration} and its proof, necessarily $w_k(0)>0$ and $w_k(\pi)<0$ and $w_k$ can not have an infinite number of zeroes before and after $a_0$. Therefore there exists  a subsequence $(w_{k_j})$, with each $k_j$ odd, such that $w_{k_j}(0)\rightarrow\infty$ and $w_{k_j}(\pi)\rightarrow-\infty$ as $j\rightarrow\infty$. Since $\arctan(f(x))=0$ for every $x\in M_+$ and $\arctan(f(x))=\pi$ for every $x\in M_-$, the sequence $(u_{k_j})$ satisfies the desired limits as $j\rightarrow\infty$.
\end{proof}

%%%%%%%%%%%%%%%%%%%%%%%%%%%%%%%%%%%%%%%%%%%%%%%%%%%%%%%%%%%%%%%%%%%%%%%%%%%%%%%
%%%%%%%%%%%%%%%%%%%%%%%%%%%%%%%%%%%%%%%%%%%%%%%%%%%%%%%%%%%%%%%%%%%%%%%%%%%%%%%

\section{Submersion with minimal fibers and the proof of Corollary \ref{Cor:YamabeProjective}.}\label{Sec:Projective}

We shall study Riemannian submersions of the form $\pi:(\S^n,g_0)\rightarrow (M^m,g)$, where $(M,g)$ is a closed Riemannian manifold with dimension $m<n$, and such that the fibers are minimal. It is well known that under these conditions, $\pi$ is a harmonic morphism with dilation $\varphi\equiv 1$, see \cite{bw} or \cite[Section 11.2]{f}. In this case, a function $v:M\rightarrow\R$ solves the equation
\begin{equation}\label{Eq:Reduced equation}
-\Delta_g v + \lambda v = \mu \vert v\vert^{p-1}v,\quad\text{on }M 
\end{equation}
with $\lambda,\mu>0$, if and only if $u=\left[ \frac{\mu}{\lambda}\right]^{\frac{1}{p-1}}v\circ \pi$ is a solution to the Yamabe problem on the sphere \eqref{Eq:Yamabe sphere} (Cf. \cite{cfp}). Observe that if we are considering the critical exponent $p_n$ for the equation \eqref{Eq:Yamabe sphere}, then this same exponent is subcritical in equation \eqref{Eq:Reduced equation} for, in this case, $m<n$ implies $p_n<p_m$.

We next show that isoparametric functions are preserved by Riemannian submersions with minimal fibers.

\begin{lemma}\label{Lemma:Isoparam functions under Riemannian submersions}
	Let $\pi:(N,h)\rightarrow(M,g)$ be a Riemannian submersion with minimal fibers. If a smooth function $\widehat{f}:M\rightarrow\R$ is isoparametric, then $f:=\widehat{f}\circ\pi:N\rightarrow\R$ is isoparametric. Conversely, if $f:N\rightarrow \R$ is isoparametric and $f(x)=f(y)$ whenever $\pi(x)=\pi(y)$, then there exists a smooth function $\widehat{f}:M\rightarrow \R$ such that $f=\widehat{f}\circ\pi$ and $\widehat{f}$ is isoparametric.
\end{lemma}	

\begin{proof}
	Suppose $\widehat{f}:M\rightarrow\R$ is isoparametric with $\vert \nabla \widehat{f}\vert_g^2=\widehat{a}(\widehat{f})$ and $\Delta \widehat{f}=\widehat{b}(\widehat{f})$. On the one hand, as $\pi:N\rightarrow M$ is a harmonic morphism with dilation $\varphi\equiv 1$ (see \cite{bw}), we have that
	\[
	\Delta_h f=\Delta_h (\widehat{f}\circ\pi)=(\Delta_g \widehat{f})\circ\pi=(\widehat{b}\circ \widehat{f})\circ\pi=\widehat{b}(\widehat{f}\circ\pi)=\widehat{b}(f).
	\]	
	
	Next, observe that
	\[
	\nabla \widehat{f} = \pi_\ast\nabla f
	\]
	because $\pi:N\rightarrow M$ is a Riemannian submersion. 
	
	Therefore,
	\begin{align*}
	&\vert \nabla f(x)\vert_h^2 = h(\nabla f(x),\nabla f(x))=g(\pi_\ast\nabla f(x),\pi_\ast\nabla f(x))\\
	&=g(\nabla \widehat{f}(\pi(x)),\nabla \widehat{f}(\pi(x)))=\vert \nabla \widehat{f}(\pi(x))\vert_g^2=a(\widehat{f}(\pi(x)))=\widehat{a}(f(x)),
	\end{align*}
	and $f$ is an isoparametric function.
	
	Now suppose that $f$ is an isoparametric function such that $f(x)=f(y)$ if $\pi(x)=\pi(y)$. Let $a,b:\R\rightarrow\R$ be such that $\vert \nabla f\vert_h^2=a(f)$ and $\Delta_h f=b(f)$. As $\pi:N\rightarrow M$ is a submersion, it is a quotient map \cite{lee} and the function $f$ passes to the quotient as an smooth function $\widehat{f}:M\rightarrow \R$ such that $f=\widehat{f}\circ\pi$. Then, as before
	\[
	a\circ \widehat{f}\circ\pi(x)=a(f(x))=\vert \nabla f(x)\vert_n^2=\vert \nabla \widehat{f}(\pi(x))\vert_g^2
	\]
	and
	\[
	b\circ \widehat{f}\circ\pi=b(f)=\Delta_h f= \Delta_h(\widehat{f}\circ\pi)=(\Delta_g \widehat{f})\circ\pi.
	\]
	Since $\pi$ is surjective, it has a right inverse and we conclude that $\vert \nabla \widehat{f}\vert_g^2=a(\widehat{f})$ and $\Delta_g \widehat{f}=b(\widehat{f})$.
\end{proof}

Let $(M,g)$ denote $(\C P^m,\tilde{g}_{0})$ or $(\H P^m,\tilde{g}_0)$, the complex and quaternionic spaces with their corresponding canonical metrics. Both of them are Einstein manifolds with constant positive scalar curvature (see \cite[Theorem 14.39]{b} and Sections 8.1 and 8.3 in \cite{p}). We will denote it by $R_{\tilde{g}}=\Lambda_m>0$. Recall that $\mathbb{C}P^m=\S^{2m+1}/\S^1$ and that $\H^m=\S^{4m+3}/SU(2)$, where $\S^1$ is the circle group and $SU(2)$ is the group of unit quaternions. We need the following lemma.

\begin{lemma}\label{Lemma:ExistenceIsoparametric}
We have the following:
\begin{enumerate}
	\item For each $m\geq 3$, there exists an $\S^1$-invariant Cartan-M\"{u}nzner polynomial $f:\S^{2m+1}\rightarrow\R$ such that the associated focal manifolds $M_{-}$ and $M_{+}$ satisfy $\dim M_{\pm}\geq 2$.
	\item For each $m\geq 3$ there exists an $SU(2)$-invariant Cartan-M\"{u}nzner polynomial $f:\S^{4m+3}\rightarrow\R$ such that the associated focal manifolds $M_{-}$ and $M_{+}$ satisfy $\dim M_\pm\geq 4$.
\end{enumerate}
\end{lemma} 

\begin{proof}
We begin with the proof of (1). In this case, we consider $\S^{2m+1}\subset\R^{2m+2}$. As $m\geq 3$, we can write $m=2+k$ with $k\in\N\smallsetminus\{0\}$ and we can decompose $\R^{2m+2}\equiv\R^{2\alpha}\times\R^{2\beta}\equiv\C^\alpha\times\C^\beta$ where $\alpha=\beta=2$ if $k=1$ and $\alpha=k$, $\beta=3$ if $k\geq 2$. Hence, we have an action of $\S^1$ by isometries given as
$\zeta (x,y):=(\zeta x,\zeta y)$, where $\zeta(x,y):=(\zeta x_1,\ldots,\zeta x_{\alpha},\zeta y_1,\ldots,\zeta y_{\beta})$, with $\zeta\in\S^1\subset\C$ and $x_i,y_j\in\C$. We can then consider the degree two Cartan-M\"{u}nzner polynomial $f:\R^{2m+2}\equiv\C^\alpha\times\C^\beta\rightarrow\R$, $f(x,y)=\vert x\vert^2 - \vert y\vert^2$, which is clearly $\S^1$-invariant. The restriction of $f$ to $\S^{2m+1}$ is an isoparametric function with focal submanifolds $M_{-}=\S^{2\alpha-1}\times \{0\}$ and $M_{+}=\{0\}\times\S^{2\beta-1}$, see \cite{cr}. If $k=1$, $\dim M_\pm=3$ while if $k\geq 2$, $\dim M_-=2k-1\geq 3$ and $\dim M_+=5$, and the lemma follows.

Now we prove (2) in a similar way. For each $m\geq 4$, write $m=4+k$ with $k\in\N\cup\{0\}$. Hence $\S^{4m+3}\subset \R^{4m+4}\equiv\R^{4\alpha}\times\R^{4\beta}\equiv\H^{\alpha}\times\H^{\beta}$, where $\alpha=\beta=4$ if $m=3$ and $\alpha=k+2$ and $\beta=3$ if $m\geq 4$. Hence we have a natural action of $SU(2)$ on $\R^{4m+4}$ given by $\zeta(x,y)=(\zeta x,\zeta y)=(\zeta x_1,\ldots,\zeta x_{\alpha},\zeta y_1,\ldots,\zeta y_{\beta})$ where $\zeta\in SU(2)$ and $x_i,y_j\in \H$. Since $\vert \zeta q\vert=\vert \zeta\vert\vert q\vert=\vert q\vert$ for each $q\in \H$ and $\zeta\in SU(2)$, we have that the Cartan-M\"{u}nzner polynomial $f:\R^{2m+2}\equiv\H^\alpha\times\H^\beta\times\R^{2m+2}\rightarrow\R$ given by $f(x,y)=\vert x\vert^2 - \vert y\vert^2$ is $SU(2)$-invariant. As before, the corresponding focal submanifolds are given by $M_{-}=\S^{4\alpha-1}\times\{0\}$ and $M_{+}=\{0\}\times\S^{\beta - 1}$. Thus, $\dim M_\pm= 7$ if $m=3$, while $\dim M_-=4(k+2)-1\geq 11$ and $\dim M_+=11$ for every $m\geq 4$, where we conclude the lemma in this case.			
\end{proof}

\begin{proof}[Proof of Corollary \ref{Cor:YamabeProjective}]
We analize the case of the quaternionic projective space, being the case for $\C P^m$ completely analogous. First we write \eqref{Eq:Yamabe projective} in the form \eqref{Eq:Reduced equation}, 
\begin{equation}\label{Eq:YamabeQuaternionic}
-\Delta_{g}v+\lambda_m v =\mu_m \vert v\vert^{p_{4m}-1}\qquad \text{ on }\H P^m,
\end{equation}
where $\lambda_m:=\frac{\Lambda_m(4m-2)}{16m-4}$ and $\mu_m:=\frac{4m-2}{16m-4}$. Let $\pi:\S^{4m+3}\rightarrow \H P^m$ be the natural projection. This map is a Riemannian submersion with minimal fibers and, so, it is a harmonic morphism \cite{bw}.

Therefore, $v$ is a solution to the Yamabe problem
\eqref{Eq:Yamabe projective} if and only if $u=\left[\frac{\mu_m}{\lambda_m}\right]^{\frac{1}{p_{4m}-1}}v\circ\pi$ is a solution to the supercritical problem on the sphere
\begin{equation}\label{Eq:SupercriticalSphereQuaternionic}
-\Delta_{g_0}u+\lambda_{m}u=\lambda_m\vert u\vert^{p_{4m}-1}u\qquad\text{on }\S^{4m+3}.
\end{equation}
By Lemma \ref{Lemma:ExistenceIsoparametric}, there exists an $SU(2)$-invariant isoparametric function $f:\S^{4m+3}\rightarrow\R$, which is the restriction of a Cartan-M\"{u}nzner polynomial and such that its corresponding focal submanifold have dimension at least $4$. Hence, if we take $\kappa = \min\{\dim M_-,\dim M_+\}\geq 4$, then
\[
p_{4m}<\frac{(4m+3-\kappa)+2}{(4m+3-\kappa)-2}
\]
and the supercritical problem \eqref{Eq:SupercriticalSphereQuaternionic} admits a sequence of nodal solutions $(u_k)$ of the form $u_k=z_k\circ f$ where $z_k$ is a solution to the problem \eqref{Eq:Subcritical Yamabe ODE} with at least $k$ zeroes. As $f$ is $SU(2)$-invariant, there exists $\widehat{f}:\H P^{m}\rightarrow\R$ such that $f=\widehat{f}\circ \pi$. Therefore $u_k=z_k\circ f=(z_k\circ\widehat{f})\circ\pi$ is a solution to \eqref{Eq:SupercriticalSphereQuaternionic}, implying that $v_k:=\left[\frac{\mu_m}{\lambda_m}\right]^{-\frac{1}{p_{4m}-1}}z_k\circ\widehat{f}$ is a solution to the Yamabe problem \eqref{Eq:Yamabe projective}. Moreover, by Lemma \ref{Lemma:Isoparam functions under Riemannian submersions}, $\widehat{f}$ is an isoparametric function and, thus, $v_k$ has isoparametric hypersurfaces in $\H P^{m}$ as level sets.
Finally, the blow-up of the sequence $(v_k)$ is a consequence of the limit \eqref{Eq:Blow-up} for the sequence $(u_k)$ at $M_-$ or at $M_+$.
\end{proof}

%%%%%%%%%%%%%%%%%%%%%%%%%%%%%%%%%%%%%%%%%%%%%%%%%%%%%%%%%
%%%%%%%%%%%%%%%%%%%%%%%%%%%%%%%%%%%%%%%%%%%%%%%%%%%%
%%%%%%%%%%%%%%%%%%%%%%%%%%%%%%%%%%%%%%%%%%%%%%%%%%%%%%%%%%%%
\appendix

\section{Energy analysis and proof of Lemma \ref{Lemma:Radius}}\label{Sec:Energy Analysis}

In this appendix we prove Lemma \ref{Lemma:Radius} for the solutions to the initial value problem \eqref{Eq:Singular forward}. An entirely similar argument will hold for the problem \eqref{Eq:Singular backward equivalent} instead of \eqref{Eq:Singular backward}. As in Section \ref{Sec:Double Shooting}, we will suppose in what follows that the multiplicities of the principal curvatures of the isoparametric family satisfy $1\leq m_{-}\leq m_{+}$ and that $p$ satisfies inequality \eqref{Eq:Key inequality}. To simplify the notation, we write equation \eqref{Eq:Subcritical radial Yamabe ODE} as
\begin{equation}\label{Eq:Yamabe ODE}
w'' + \frac{h(r)}{\sin r}w' + g(w)=0\qquad\text{in }[0,\pi],
\end{equation}
where $g(t):=\frac{\lambda}{\ell^2}(\vert t\vert^{p-1}t - t)$. 
 For the initial conditions $w(0)=d$ and $w'(0)=0$, let $w_d$ be the unique solution to problem \eqref{Eq:Yamabe ODE} on $[0,a_0]$. Define the energy function
 \[
 E(r,d):=\frac{(w_d'(r))^2}{2}+G(w_d(r)),
 \]
 where
 \[G(t):=\int_0^tg(s)ds=\frac{\lambda}{ \ell^2}\left(\frac{\vert t\vert^{p+1}}{p+1}-\frac{t^2}{2}\right).
 \] 
 Observe that $E$ is nonincreasing on  $r\in[0,a_0]$, for
 \[
 E'(r,d)=-\frac{h(r)}{\sin r}(w'_d(r))^2.
 \]
The aim of this section is to prove that $E(a_0,d)\rightarrow\infty$ uniformly in $[0,a_0]$ as $d\rightarrow\infty$, for this will immediately imply Lemma \ref{Lemma:Radius}. Since the proof is long and technical, we first sketch it in few lines, following the ideas given in \cite{cf} and \cite{ck}.

First, in Lemma \ref{Lemma:Existence r_0}  we prove the existence of the value $r_0(d)$ for which $w_d(r_0)=\kappa d$ and $d\geq w_d(r)\geq \kappa d$ for a suitable $\kappa\in(0,1)$ and for every $r \in [0, r_0]$.  Since we will show that $r_0(d)\rightarrow 0$, we need to prove the existence of fixed $T>0$, independent of $d$, such that $E(r,d)\rightarrow\infty$ uniformly in $[0,T]$ as $d\rightarrow\infty$. To see this, we establish a version of the Pohozaev identity \cite{poh} for equation \eqref{Eq:Yamabe ODE}, generalizing the ones given in \cite{cf,ck}. This identity together with the properties of $r_0$ and inequality \eqref{Eq:Key inequality} will imply the existence of such $T>0$. Finally, we  prove that $E(r,d)\geq e^{-2T}E(T,d)+C$ for every $r\in[T,a_0]$, where $C$ is a constant independent of $d$ and $r$. The last inequality implies the desired uniform limit.

We begin with the existence of $r_0$.

\begin{lemma}\label{Lemma:Existence r_0}
For each $\kappa\in(0,1)$, there exists $\hat{D}_1=\hat{D}_1(\kappa)\geq 1$ such that for every $d\geq \hat{D}_1$, $\kappa d\geq 1$ and there is $r_0=r_0(d)\in(0,\pi)$ such that
	\begin{equation}
	w_d(r_0)=\kappa d\qquad\kappa d\leq w_d(r)\leq d,\text{ for every }r\in[0,r_0]\quad\text{and}\quad\lim_{d\rightarrow\infty}r_0(d)=0.
	\end{equation}
\end{lemma}
\begin{proof}
	As inequality \eqref{Eq:Key inequality} implies \eqref{Eq: Subcriticality ODE}, Theorem \ref{Prop:prescribed zeroes} gives the existence of $D_1\geq 1$ such that $w_d$ has a zero in $[0,a_0]$. Take $\hat{D}_1$ so that $\kappa d\geq 1$ for every $d\geq \hat{D}_1$ and for each $d\geq \hat{D}_1$, denote by $r_1(d)$  the first zero of $w_d$. Then $w_d$ is strictly decreasing in $[0,r_1]$ and $r_1$ depends continuously on $d$. Therefore, there exists $r_0(d)\in(0,r_1(d))$ such that $\kappa d=w_d(r_0)\leq w_d(r)\leq w_d(0)=d$ for every $r\in[0,r_0]$. By Lemma \ref{Lemma:Limit zeroes} we conclude that $r_0(d)\rightarrow 0$ as $d\rightarrow\infty$.
\end{proof}

Recall $\ell$ denotes the number of distinct principal curvatures of the isoparametric hypersurfaces associated to the isoparametric function $f$ on the sphere, and let $1\leq m_{-}\leq m_{+}\leq n-1$ be the (possibly equal) multiplicities of the principal curvatures. In this situation, $\frac{n-1}{\ell}=\frac{m_{-}+m_{+}}{2}$ and an integrating factor for equation \eqref{Eq:Yamabe ODE} is given by
\[
q(r):=(\sin r)^{\frac{n-1}{\ell}} (\tan r/2)^{-\frac{m_{+}-m_{-}}{2}}=2^{\frac{m_{-}+m_{+}}{2}}(\sin r/2)^{m_{-}}(\cos r/2)^{m_{+}}.
\]
Therefore, equation \eqref{Eq:Yamabe ODE} can be written in divergence form as
\begin{equation}\label{Eq:Div Yamabe ODE}
(qw')' + q(r)g(w)=0\qquad\text{in }[0,\pi]
\end{equation}
and $q$ satisfies
\begin{equation}\label{Eq:Property integral factor}
\frac{q'(r)}{q(r)}=\frac{h(r)}{\sin r}
\end{equation}

Define $\zeta:[0,a_0]\rightarrow\mathbb{R}$ by
\[
\zeta(r):=\left\{\begin{tabular}{cc}$q(r)\int_{r}^{a_0}q^{-1}(s) ds$, & if $r\neq 0$,\\ 
\!\!\!\!\!\!\!\!\!\!\!\!\!\!\!
\!\!\!\!\!\!\!\!\!\!\!\!
\!\!\!\!\!\!$0$, & if $r=0$.\end{tabular}\right..
\]

Observe this function is continuous at $r=0$, for
\begin{equation}
\lim_{r\rightarrow 0}\zeta(r)=0
\end{equation}
by L'H\^{o}pital's rule and \eqref{Eq:Property integral factor}. Notice also that
\begin{equation}\label{Eq:Derivative zeta}
\zeta'(r)=\frac{h(r)}{\sin r}\zeta(r)-1.
\end{equation}

We next derive a useful Pohozaev-like identity.

\begin{lemma}\label{Lemma:Pohozaev}
	If $w_d$ is the solution to \eqref{Eq:Yamabe ODE} in $[0,a_0]$ with initial conditions $w_d(0)=d$ and $w_d'(0)=0$, then
	\begin{equation}\label{Eq:Pohozaev}
	P(r,d):=qw_dw_d'+2q\zeta E(r,d)=\int_0^rq\left\{G(w_d)\zeta\left[4\frac{h(s)}{\sin s}-2\right]-g(w_d)w_d\right\}ds.
	\end{equation}
\end{lemma}

\begin{proof}
	On the one hand, multiplying equation \eqref{Eq:Yamabe ODE} by $qw_d$, integrating from 0 to $r\leq a_0$ and integrating by parts we obtain
	\begin{equation}\label{Eq:Pohozaev (1)}
	qw_dw_d'-\int_0^{r}q(w_d')^2 ds + \int_0^rqg(w_d)w_d=0
	\end{equation}
	because $w_d'(0)=0$ and $q'=q\frac{h(r)}{\sin r}$.
	
	On the other hand, multiplying equation \eqref{Eq:Yamabe ODE} by $q\zeta w_d'$, integrating from 0 to $r$, using integration by parts and \eqref{Eq:Derivative zeta} we have that
	\begin{equation}\label{Eq:Pohozaev (2)}
	q\zeta(w_d')^2+\int_0^rq(w_d')^2ds+2\int_0^rq\zeta g(w_d)w_d'ds=0.
	\end{equation}
	But also, integration by parts yields
	\[
	\int_0^rq\zeta g(w_d)w_d'ds=q\zeta G(w_d)-\int_0^r qG(w_d)\zeta\left[ 2\frac{h(s)}{\sin s}-1 \right]ds
	\]
	Thus, adding \eqref{Eq:Pohozaev (1)} and \eqref{Eq:Pohozaev (2)}, and using the above equality, identity \eqref{Eq:Pohozaev} follows.
\end{proof}

Observe that the derivative of $w_d$ does not appear in the right hand side of the identity, while the energy appears explicitly on the left hand side. Observe also that if $r\in[0,a_0]$ is such that $P(r,d)\rightarrow\infty$ as $d\rightarrow\infty$, then also $E(r,d)\rightarrow\infty$ as $d\rightarrow\infty$. In this direction, we state and prove the following Lemma.  

\begin{lemma}\label{Lemma:LimitPohozaev}
There exists $T>0$ small enough and fixed such that
\begin{equation}\label{Eq:LimitPohozaev}
\lim_{d\rightarrow\infty}P(r,d)=\infty
\end{equation}
uniformly in $[r_0,T]$.
\end{lemma}

The proof of this lemma is long and technical, and will take the following four pages. As will continue with the argument after its proof with Lemma  \ref{Lemma:Energy}, the reader may skip it in a first reading.

\emph{Proof.}  The proof uses strongly that inequality \eqref{Eq:Key inequality} holds true. For the reader convenience, we divide it into three steps.

We begin with some estimates of $r_0(d)$ in terms of the initial condition $d$. 

\begin{step}\label{Lemma:Estimates r_0}
	Let $\kappa\in(0,1)$ and let $\hat{D}_1:=\hat{D}_1(\kappa)$ as in Lemma \ref{Lemma:Existence r_0}. Then, for $d\geq \hat{D}_1$, there exist $\kappa_1,\kappa_2,\kappa_3>0$ independent of $d$ such that
	\begin{equation}
	\kappa_1 e^{\frac{d}{g(\kappa g)}} \leq \cos r_0/2 \quad\text{and}\quad\kappa_2\left[\frac{d}{g(d)}\right]^{\frac{1}{2}}\leq\sin r_0/2\leq\kappa_3\sqrt{\frac{d}{g(\kappa d)}}.
	\end{equation}
\end{step}

\begin{proof}[Proof of Step \ref{Lemma:Estimates r_0}]
	First observe that the integrating factor $q$ satisfies, for any $R\in(0,\pi)$,  the following estimates
	\begin{equation}\label{Eq:Estimates integral factor}
	\frac{2^{\frac{m_{-}+m_{+}}{2}}}{m_{-}+1}(\cos R/2)^{m_{+}-1}(\sin R/2)^{m_{-}+1}\leq\int_0^Rq(r)dr\leq\frac{2^{\frac{m_{-}+m_{+}}{2}}}{m_{-}+1}(\sin R/2)^{m_{-}+1}.
	\end{equation}
	Now, let $\kappa\in(0,1)$ and $\hat{D}_1$ be as in the hypotheses of Lemma \ref{Lemma:Existence r_0}. As $1\leq\kappa d\leq w_d(r)\leq d$ in $[0,r_0]$ when $d\geq \hat{D}_1$, and as $g$ is nondecreasing in $[1,\infty)$, then $0\geq-g(\kappa d)\geq -g(w_d)\geq-g(d)$ in $[0,r_0]$. Now, integrating equation \eqref{Eq:Div Yamabe ODE} from 0 to $r<r_0$ and recalling that $w_d'(0)=0$, we get
	\begin{align*}
	2^{\frac{m_{-}+m_{+}}{2}}(\sin r/2)^{m_{-}}(\cos r/2)^{m_{+}}w_d'(r)&=q(r)w_d'(r)=qw_d'\big|_{0}^{r}=\int_0^r(qw_d')'ds\\
	&=-\int_0^{r}qg(w_d)\leq -g(\kappa d)\int_0^rq(s)ds\\
	&\leq-\frac{2^{\frac{m_{-}+m_{+}}{2}}}{m_{-}+1}g(\kappa d)(\cos r/2)^{m_{+}-1}(\sin r/2)^{m_{-}+1}.
	\end{align*}
	Hence
	\[
	w_d'\leq-\frac{g(\kappa d)}{m_{-}+1}(\cos r/2)^{-1}\sin r/2.
	\]
	Integrating over $[0,r_0]$ we obtain
	\begin{align*}
	(\kappa-1)d&=w_d(r_0)-w_d(0)\leq -\frac{g(\kappa d)}{m_{-}+1}\int_{0}^{r_0}(\cos r/2)^{-1}\sin r/2\;dr\\
	&=\frac{2g(\kappa d)}{m_{-}+1}\int_{1}^{\cos r_0/2}x^{-1}dx\\
	&=\frac{2g(\kappa d)}{m_{-}+1}\ln \cos r_0/2.
	\end{align*}
	Therefore
	\[
	\kappa_1e^{\frac{d}{g(\kappa d)}}\leq \cos r_0/2,\quad\text{with}\quad 0<\kappa_1:=e^{\frac{(\kappa-1)(m_{-}+1)}{2}}<1
	\]
	Similarly we obtain the second estimate: Integrating equation \eqref{Eq:Div Yamabe ODE} from 0 to $r<r_0$ we have that
	\begin{align*}
	2^{\frac{m_{-}+m_{+}}{2}}(\sin r/2)^{m_{-}} w_d'\geq q(r)w_d'(r)&=-\int_0^r qg(w_d) ds\\
	&\geq -g(d)\int_0^r q(s) ds\geq -\frac{2^{\frac{m_{-}+m_{+}}{2}}}{m_{-}+1} (\sin r/2)^{m_{-}+1}.
	\end{align*}
	So $w_d'\geq -\frac{g(d)}{m_{-}+1}\sin r/2$. Integrating over $[0,r_0]$ we get that
	\[
	(\kappa-1)d=w_d(r_0)-w_d(0)\geq-\frac{g(d)}{m_{-}+1}\int_{0}^{r_0}\sin r/2\;dr=\frac{2g(d)}{m_{-}+1}\left[\cos r_0/2 -1 \right]
	\]
	Noticing that $0\leq\cos r/2\leq 1$ in $[0,\pi]$ implies $\sin^2r/2=1-\cos^2 r/2\geq 1-\cos r/2$, we write
	\[
	\frac{(1-\kappa)(m_{-}+1)}{2}\frac{d}{g(d)}\leq 1-\cos r_0/2\leq \sin^2 r_0/2,
	\]
	where we conclude that
	\[
	\kappa_2\left[\frac{d}{g(d)}\right]^{1/2}\leq \sin r_0/2,\quad\text{with}\quad \kappa_2:=\left[\frac{(1-\kappa)(m_{-}+1)}{2}\right]^{1/2}>0.
	\]
	Next, for the third one, since $-(\cos r/2)^{m_{+}+1}\geq -(\cos r/2)^{m_{+}-1}$, we get in the same fashion that
	\begin{align*}
	2^{\frac{m_{-}+m_{+}}{2}}&(\sin r/2)^{m_{-}}(\cos r/2)^{m_{+}}w_d'(r)=-\int_0^{r}qg(w_d)\leq -g(\kappa d)\int_0^rq(s)ds\\
	&\leq-\frac{2^{\frac{m_{-}+m_{+}}{2}}}{m_{-}+1}g(\kappa d)(\cos r/2)^{m_{+}-1}(\sin r/2)^{m_{-}+1}\\
	& \leq-\frac{2^{\frac{m_{-}+m_{+}}{2}}}{m_{-}+1}g(\kappa d)(\cos r/2)^{m_{+}+1}(\sin r/2)^{m_{-}+1}.
	\end{align*}
	So
	\[
	w_d'(r)\leq -\frac{g(\kappa d)}{m_{-}+1}\cos r/2 \sin r/2.
	\]
	Integrating from $0$ to $r_0$ we obtain
	\[
	(\kappa -1)d=\int_0^{r_0}w_d'(s)ds\leq-\frac{g(\kappa d)}{m_{-}+1}\int_0^{r_0}\cos s/2\;\sin s/2\; ds=-\frac{g(\kappa d)}{m_{-}+1}\sin^2r_0/2.
	\]
	Hence
	\[
	\kappa_3^2\frac{d}{g(\kappa d)}\geq\sin^2 r_0/2,\quad\text{with}\quad\kappa_3^2:=(m_{-}+1)(1-\kappa)>0.
	\]
\end{proof}

\begin{step}\label{Step:kappa}
	There exist $\kappa\in(0,1)$ and $\theta>0$ such that
	\begin{equation}\label{Eq:Crucial limit}
	\lim_{d\rightarrow\infty}\left[ \theta G(\kappa d)-g(d)d \right]\left[ \frac{d}{g(d)} \right]^{\frac{m_{-}+1}{2}}=\infty
	\end{equation}
\end{step}

\begin{proof}[Proof of Step \ref{Step:kappa}]
	As $\lim_{r\rightarrow 0}\zeta(r)=0$ and $\lim_{r\rightarrow 0}h(r)=m_{-}$, writing $\sin r=2\sin\frac{r}{2}\cos\frac{r}{2}$ and using L'H\^{o}pital's rule we have that
	\begin{align}\label{Eq:Limit zeta}
	\lim_{r\rightarrow 0} \zeta\left[4\frac{h(s)}{\sin s}-2\right]&=4m_{-}\lim_{r\rightarrow0}\frac{q\int_r^{a_0}q^{-1}(s)ds}{\sin s}\nonumber\\
	&=2m_{-}\lim_{r\rightarrow0}\frac{\int_r^{a_0} (\sin s/2)^{-m_{-}}(\cos s/2)^{-m_{+}}\;ds}{(\sin r/2)^{1-m_{-}}(\cos r/2)^{1-m_{+}}}\nonumber\\
	&=2m_{-}\lim_{r\rightarrow0}\frac{-1}{\frac{1-m_{-}}{2}\cos^{2} r/2-\frac{1-m_{+}}{2}\sin^2r/2}\nonumber\\
	&=\frac{4m_{-}}{m_{-}-1}.
	\end{align}
	Inequality \eqref{Eq:Key inequality} yields that $0<\frac{4m_{-}}{m_{-}-1}-(p+1)$. Therefore, we can fix $0<\theta:=\frac{4m_{-}}{m_{-}-1}-\varepsilon$ and $\kappa:=(1-\delta)^{1/(p+1)}\in(0,1)$ with $\varepsilon,\delta>0$ small enough so that
	\[
	\frac{\theta-(p+1)}{p+1}>\frac{\theta\kappa^{p+1}-(p+1)}{p+1}>0.
	\]
	
	Hence, the functions $\theta G(\kappa t)-tg(t)=\lambda\left[\frac{\theta\kappa^{p+1}-(p+1)}{p+1}\vert t\vert^{p+1}-\frac{\theta\kappa^2-2}{2}\vert t\vert^2\right]$ and $G(t)-tg(t)$ are bounded from below. Using this, inequality \eqref{Eq: Subcriticality ODE} and that $1<p$, we have that	\begin{align}
	&\lim_{d\rightarrow\infty}\left[ \theta G(\kappa d)-g(d)d \right]\left[ \frac{d}{g(d)} \right]^{\frac{m_{-}+1}{2}}\nonumber\\
	&=\lim_{d\rightarrow\infty}\lambda^{\frac{1-m_{-}}{2}}\frac{\left(\frac{\theta\kappa^{p+1} -(p+1)}{p+1}\right)\vert d\vert^{(p+1)}-\left(\frac{\theta\kappa^2 -2}{2}\right)\vert d\vert^{2}}{\left[ \vert d\vert^{p-1}-1 \right]^{\frac{m_{-}+1}{2}}}\nonumber\\
	&=\lim_{d\rightarrow\infty}\lambda^{\frac{1-m_{-}}{2}}\frac{\left(\frac{\theta\kappa^{p+1} -(p+1)}{p+1}\right)\vert d\vert^{(p+1)-\frac{(p-1)(m_{-}+1)}{2}}-\left(\frac{\theta\kappa^2 -2}{2}\right)\vert d\vert^{2-\frac{(p-1)(m_{-}+1)}{2}}}{\left[ 1-\vert d\vert^{-(p-1)} \right]^{\frac{m_{-}+1}{2}}}\nonumber\\
	&=\infty\nonumber
	\end{align}
\end{proof}

\begin{step}\label{Step:Existence T}
	There exists $T>0$ such that
	\[
	\lim_{d\rightarrow\infty}P(r,d)=\infty
	\]
	uniformly in $[r_0,T]$
\end{step}

\begin{proof}[Proof of Step \ref{Step:Existence T}]	
	For $\kappa$ and $\theta$ as above, consider $\hat{D}_1\geq 1$ as in Lemma \ref{Lemma:Existence r_0}. By limit \eqref{Eq:Limit zeta}, we can choose $T\in (0,a_0)$ small enough so that
	\[
	\zeta\left[4\frac{h(s)}{\sin s}-2\right]\geq \theta,\qquad r\in[0,T].
	\]
	By Step \ref{Lemma:Estimates r_0}, $r_0(d)\rightarrow 0$ when $d\rightarrow\infty$ and we can choose $\hat{D}_2\geq \hat{D}_1$ such that $r_0(d)<T$ for every $d\geq \hat{D}_2$. Since for every $d\geq \hat{D}_2$, we have that $1\leq\kappa d\leq w_d(r)\leq d$ for every $r\in[0,r_0]$, and since the functions $G(t)$ and $tg(t)$ are nondecreasing when $t\geq 1$, it follows that $G(w_d)\geq G(\kappa d)\geq 0$ and that $-g(w_d)w_d\geq -g(d)d$. Hence
	\[G(w_d)\zeta\left[\frac{4h(r)}{\sin r}-2\right]-g(w_d)w_d\geq \theta G(\kappa d)-g(d)d>0,\]
	for every $r\in[0,r_0]$ and every $d\geq \hat{D}_3$,
	where $\hat{D}_3\geq \hat{D}_2$ is such that $G(\kappa d)-g(d)d>0$ for every $d\geq \hat{D}_3$.
	
	First we prove the following limit
	\[
	\lim_{d\rightarrow\infty}P(r_0,d)=\infty
	\]
	Indeed, since $d\geq\hat{D}_3$, the estimates \eqref{Eq:Estimates integral factor} and the ones obtained in Step \ref{Lemma:Estimates r_0} yield that
	\begin{align*}
	P(r_0,d)&=\int_0^{r_0}q\left\{ G(w_d)\zeta\left[ \frac{4h(s)}{\sin s}-2 \right] - g(w_d) \right\} ds\\
	&\geq \left[ \theta G(\kappa d)-g(d)d \right]\int_0^{r_0}q ds\\
	&\geq \frac{2^{\frac{m_{-}+m_{+}}{2}}}{m_{-}+1}\left[ \theta G(\kappa d)-g(d)d \right](\cos r_0/2)^{m_{+}-1}(\sin r_0/2)^{m_{-}+1}\\
	&\geq \frac{2^{\frac{m_{-}+m_{+}}{2}}\kappa_1^{m_{+}-1}\kappa_2^{m_{-}+1}}{m_{-}+1}\left[ \theta G(\kappa d)-g(d)d \right]e^{(m_{+}-1)\frac{d}{g(d)}}\left[ \frac{d}{g(d)} \right]^{\frac{m_{-}+1}{2}}
	\end{align*}
	and since $e^{(m_{+}-1)\frac{d}{g(d)}}\rightarrow 1$ as $d\rightarrow\infty$, \eqref{Eq:Crucial limit} implies $P(r_0,d)\rightarrow\infty$ as $d\rightarrow\infty$ as we wanted.
	
	Next, as $\theta G(\kappa d)-g(d)d$ is bounded from below, there exists $M<0$ such that $\theta G(\kappa d)-dg(d)\geq M$ for every $d\in\mathbb{R}$. Then, for every $r\in[r_0,T]$ we have that
	\begin{align*}
	P(r,d)&=P(r_0,d)+\int_{r_0}^rq\left\{ G(w_d)\zeta\left[ \frac{4h(s)}{\sin s}-2 \right] - g(w_d) \right\} ds\\
	&\geq P(r_0,d)+\left[ \theta G(\kappa d)-g(d)d \right]\int_{r_0}^rq(s)ds\\
	&\geq P(r_0,d)+\frac{2^{\frac{m_{-}+m_{+}}{2}}M}{m_{-}+1}\left[(\sin r/2)^{m_{-}+1}-(\sin r_0/2)^{m_{-}+1}\right]\\
	&\geq P(r_0,d)+\frac{2^{\frac{m_{-}+m_{+}+2}{2}}M}{m_{-}+1}
	\end{align*}
	and as the last constant does not depend on $r\in[r_0,T]$ and $d\geq \hat{D}_1$, it follows that $\lim_{d\rightarrow\infty}P(r,d)=\infty$ uniformly on $[r_0,T]$, concluding the proof of Lemma \ref{Lemma:LimitPohozaev}. 
	
\end{proof}

We can now prove the uniform convergence of the energy function.

\begin{lemma}\label{Lemma:Energy}
	\begin{equation}\label{Eq:Limit Energy}
	\lim_{d\rightarrow\infty}E(r,d)=\infty,\qquad\text{uniformly in }[0,a_0]
	\end{equation}
\end{lemma}

\begin{proof}
Take $T$ as in the previous lemma. Then clearly the limit \eqref{Eq:LimitPohozaev} implies that $\lim_{d\rightarrow\infty}E(r,d)=\infty$ uniformly in $[r_0,T]$. Now we show that $E(r,d)$ also converges uniformly in $[0,r_0]$ and in $[T,a_0]$ as $d\rightarrow\infty$. 

For the first one, let $\kappa\in(0,1)$ be as in Step \ref{Step:kappa} of the proof of Lemma \ref{Lemma:LimitPohozaev} below and consider $\hat{D}_1(\kappa)>1$ as in Lemma \ref{Lemma:Existence r_0}. Then, for every $d\geq \hat{D}_1(\kappa)$ and every $r\in[0,r_0]$, we have that $1\leq\kappa d\leq w_d(r)\leq d$, which implies that $G(w_d(r))\geq G(\kappa d)$. Since $G(\kappa d)\rightarrow\infty$ as $d\rightarrow\infty$ and since
	\[
	E(r,d)=\frac{(w_d')^2}{2}+G(w_d)\geq G(w_d)\geq G(\kappa d),\qquad\text{for every }r\in[0,r_0]
	\]
we conclude that $E(r,d)\rightarrow\infty$ as $d\rightarrow\infty$ uniformly on $[0,r_0]$.
	
Finally, define $\widehat{E}(r,d)=E(r,d)-K$, where $K<0$ is a lower bound for $G(d)$. As $h(r)\geq0$ in $[T,a_0]$, being $a_0$ the unique zero of this function, by continuity we have that $\tau:=\max_{r\in[T,a_0]}\frac{h(r)}{\sin r}>0$. Hence
	\begin{align*}
	\widehat{E}'(r,d)&=E'(r,d)=-\frac{h(r)}{\sin r}(w_d')^2\geq -2\tau\frac{(w_d')^2}{2}+2\tau K-2\tau K\\
	&\geq -2\tau\frac{(w_d')^2}{2}+2\tau K-2\tau G(w_d) = -2\tau\left[\frac{(w_d')^2}{2}+G(w_d)-K \right]\\
	&=-2\tau\widehat{E}(r,d).
	\end{align*}
	Integration on $[T,r]$ yields $\widehat{E}(r,d)\geq e^{-2\tau T}\widehat{E}(T,d)$ for every $r\in[T,a_0]$. Since $E(T,d)\rightarrow\infty$ as $d\rightarrow\infty$, we get that $E(r,d)\rightarrow\infty$ uniformly in $[T,a_0]$ as $d\rightarrow\infty$ and the lemma follows.
\end{proof}

\begin{proof}[Proof of Lemma \ref{Lemma:Radius}]
If $\lim_{d\rightarrow 0}\rho(r,d)=\infty$ is not true, then $w_d$ and $w_d'$ are bounded as $d\rightarrow\infty$, contradicting \eqref{Eq:Limit Energy}.
	
Now, if $w_c(r)$ is a solution \eqref{Eq:Yamabe ODE} with initial conditions $w(\pi)=c$ and $w'(\pi)=0$, then, as it was mentioned is Section \ref{Sec:Double Shooting}, the function $\omega_c(r):=w_c(\pi-r)$ is a solution to the equivalent problem \eqref{Eq:Singular backward equivalent}. As $h$ and $\widetilde{h}$ have the same properties interchanging $m_{-}$ and $m_{+}$ and taking $\widetilde{a}_0:=\pi-a_0$ instead of $a_0$, Lemmas \ref{Lemma:Existence r_0}-\ref{Lemma:Energy} hold true for the energy $\widetilde{E}(r,c):=\frac{(\omega'_c(r))^2}{2}+G(\omega_c(r))$, $r\in[0,\widetilde{a}_0]$, because of inequality \eqref{Eq:Key inequality}. Therefore $\lim_{\vert c\vert\rightarrow\infty}\widetilde{E}(r,c)=\infty$ uniformly in $[0,\widetilde{a}_0]$ and $\vert (\omega_c(r),\omega'_c(r))\vert\rightarrow\infty$ uniformly in $[0,\widetilde{a}_0]$ as $\vert c\vert\rightarrow\infty$, concluding the proof of the lemma.
\end{proof}

%%%%%%%%%%%%%%%%%%%%%%%%%%%%%%%%%%%%%
%%%%%%%%%%%%%%%%%%%%%%%%%%%%%%%%%%%%%

%%%%%%%%%%%%%%%%%%%%%%%%%%%%%%%%%%%%%%%%%%%%%
%%%%%%%%%%%%%%%%%%%%%%%%%%%%%%%%%%%%%%%%%%%

\end{document}